\documentclass{amsart}
\usepackage{amssymb}

\input xy
\xyoption{all}

\newcommand{\IR}{\mathbb R}
\newcommand{\IN}{\mathbb N}
\newcommand{\II}{\mathbb I}
\newcommand{\e}{\varepsilon}

\newcommand{\U}{\mathcal U}

\newtheorem{theorem}{Theorem}
\newtheorem{proposition}{Proposition}

\newtheorem{corollary}{Corollary}

\theoremstyle{definition}
\newtheorem{example}{Example}

\title{Topological and ditopological unosemigroups}

\author{Taras Banakh and Iryna Pastukhova}
\address{Ivan Franko National University of Lviv, Ukraine}
\email{t.o.banakh@gmail.com, irynkapastukhova@gmail.com}
\thanks{The first author has been partially financed by NCN grant  DEC-2011/01/B/ST1/01439.}

\begin{document}

\begin{abstract} In this paper we introduce and study a new topologo-algebraic structure called a (di)topological unosemigroup.
This is a topological semigroup endowed with continuous unary operations of left and right units
(which have certain continuous division property called the dicontinuity).
We show that the class of ditopological unosemigroups contains all topological groups,
 all topological semilattices, all uniformizable topological unosemigroups, all compact topological unosemigroups, and is closed
under the operations of taking subunosemigroups, Tychonoff
product, reduced product, semidirect product, and the Hartman-Mycielski
extension.\end{abstract} \maketitle

\section{Introduction}

In this paper we shall introduce and study a new topologo-algebraic structure called a [di]topological (left, right) unosemigroup.
This is a topological semigroup endowed with continuous unary operations of (left, right) units [having certain continuous
division property called the dicontinuity]. Introducing topological unosemigroups was motivated by the problem of generalization
of Hryniv's Embedding Theorems \cite{Hryn} (on embeddings of Clifford compact topological inverse  semigroups into products of cones
over compact topological groups) beyond the class of compact topological inverse semigroups.

Topological and ditopological (left, right) unosemigroups will be introduced in Sections~\ref{s:tu} and \ref{s:di}.
In Sections~\ref{s3} and \ref{s4} we shall study the class of ditopological unosemigroups and shall show that it
contains all topological groups, all topological semilattices, and all compact Hausdorff topological unosemigroups.
In Section~\ref{s:Op} we shall introduce and study some operations over topological (left, right) unosemigroups.

\section{Topological unosemigroups}\label{s:tu}

In this section we shall introduce the notion of a topological (left, right) unosemigroup.

Let $S$ be a semigroup (i.e., a non-empty set $S$ endowed with an associative binary operation $\cdot:S\times S\to S$).
An unary operation $\lambda:S\to S$ is called a {\em left unit operation} on $S$ if $\lambda(x)\cdot x=x$ for all $x\in S$.
A semigroup $S$ endowed with a left unit operation $\lambda:S\to S$ is called a {\em left unosemigroup}.

By analogy we can introduce right unosemigroups. Namely, an unary
operation $\rho:S\to S$ on a semigroup $S$ is called a {\em right
unit operation} on $S$ if $x\cdot \rho(x)=x$ for all $x\in S$. A
{\em right unosemigroup} is a semigroup $S$ endowed with a right unit
operation $\rho:S\to S$.

A {\em unosemigroup} is a semigroup $S$ endowed with a left unit operation $\lambda:S\to S$ and a right  unit operation $\rho:S\to S$.

In should be mentioned that the class of left (resp. right) unosemigroups contains the class of domain (resp. range) semigroups, considered in \cite{DJS}.
\smallskip

Now we introduce topological versions of these notions.

A {\em topological semigroup} is a topological space $S$ endowed with a continuous associative binary operation $\cdot:S\times S\to S$.
A {\em topological left unosemigroup} is a topological semigroup $S$ endowed with a continuous left unit operation $\lambda:S\to S$.
A {\em topological right unosemigroup} is a  topological semigroup $S$ endowed with a continuous right unit operation $\rho:S\to S$.
A {\em topological unosemigroup} is a  topological semigroup $S$ endowed with a continuous left unit operation $\lambda:S\to S$ and
a continuous right unit operation $\rho:S\to S$.

Topological (left, right) unosemigroups are objects of the category whose morphisms are continuous (left, right) unosemigroup homomorphisms.
A function $h:X\to Y$ between two topological left unosemigroups $(X,\lambda_X)$ and $(Y,\lambda_Y)$ is called a {\em left unosemigroup homomorphism}
if $h$ is a semigroup homomorphism and $h(\lambda_X(x))=\lambda_Y(h(x))$ for all $x\in X$.  By analogy we can define {\em right unosemigroup homomorphisms} between topological right unosemigroups and {\em unosemigroup homomorphisms} between topological unosemigroups.

\section{Ditopological unosemigroups}\label{s:di}

To introduce {\bf di}topological (left, right) unosemigroups, we first
introduce two kinds of {\bf di}vision operations on each
semigroup. Namely, for two points $a,b$ of a semigroup $S$
consider the sets
$$a\leftthreetimes b=\{x\in S:ax=b\}\mbox{ \ \ and \ \ }b\rightthreetimes a=\{x\in S: xa=b\}$$
which can be thought as the results of left and right division of $b$ by $a$.
Respectively, for two subsets $A,B\subset S$ the sets
$$A\leftthreetimes B=\bigcup_{(a,b)\in A\times B}a\leftthreetimes b\mbox{ \ \ and \ \ }B\rightthreetimes A=\bigcup_{(a,b)\in A\times B}b\rightthreetimes a$$
can be thought as the results of left and right division of $B$ by $A$.

A continuous left unit operation $\lambda:S\to S$ on a topological semigroup $S$ is called {\em dicontinuous}
at a point $x\in S$ if for every neighborhood $O_x\subset S$ of $x$ there are neighborhoods $U_x\subset S$ and $W_{\lambda(x)}\subset \lambda(S)$
of the points $x$ and $\lambda(x)$ in $S$ and $\lambda(S)$ respectively, such that
$$(W_{\lambda(x)}\leftthreetimes U_x)\cap \lambda^{-1}(W_{\lambda(x)})\subset O_x.$$
A left unit operation $\lambda:S\to S$ is defined to be {\em dicontinuous} if it is dicontinuous at each point $x\in X$.

A {\em ditopological left unosemigroup} is a topological semigroup $S$ endowed with a dicontinuous left unit operation $\lambda:S\to S$.

A trivial (but important) example of a ditopological left unosemigroup is an idempotent topological left unosemigroup. A (topological) left
unosemigroup $(S,\lambda)$ is called {\em idempotent} if $\lambda(x)=x$ for all $x\in S$. In this case $x=\lambda(x)\cdot x=xx$, which means that $S$
is an idempotent semigroup.

\begin{proposition}\label{p:LI} Each idempotent topological left unosemigroup $(S,\lambda)$ is ditopological.
\end{proposition}

\begin{proof} Given a point $x\in S$ and a neighborhood $O_x\subset S$ of $x$ put $U_x=W_{\lambda(x)}=O_x$ and observe that
$$(W_{\lambda(x)}\leftthreetimes U_x)\cap \lambda^{-1}(W_{\lambda(x)})\subset W_{\lambda(x)}=O_x,$$
which means that the left unit operation $\lambda$ is dicontinuous at $x$ and hence the topological left unosemigroup $(S,\lambda)$ is dicontinuous.
\end{proof}

By analogy we can introduce the right and two-sided versions of these notions.

Namely, a continuous right unit operation $\rho:S\to S$ on a topological semigroup $S$ is called {\em dicontinuous} if for every point $x\in S$
and every neighborhood $O_x\subset S$ of $x$ there are neighborhoods $U_x\subset S$ and $W_{\rho(x)}\subset\rho(S)$ of $x$ and $\rho(x)$
such that $$(U_x\rightthreetimes W_{\rho(x)})\cap \rho^{-1}(W_{\rho(x)})\subset O_x.$$
A {\em ditopological right unosemigroup} is a topological semigroup $S$ endowed with a dicontinuous right unit operation $\rho:S\to S$.

A (topological) right unosemigroup $(S,\rho)$ is called {\em idempotent} if $\rho(x)=x$ for all $x\in X$. By analogy with Proposition~\ref{p:LI} we can prove:

\begin{proposition}\label{p:RI} Each idempotent topological right unosemigroup $(S,\lambda)$ is ditopological.
\end{proposition}

A {\em ditopological unosemigroup} is a topological semigroup $S$ endowed
with a dicontinuous left unit operation $\lambda:S\to S$ and
dicontinuous right unit operation $\rho:S\to S$. So, a
ditopological unosemigroup carries the structures of a ditopological left
and right unosemigroups.

A (topological) unosemigroup $(S,\lambda,\rho)$ is called {\em idempotent} if $\lambda(x)=\rho(x)=x$ for all $x\in X$.
Propositions~\ref{p:LI} and \ref{p:RI} imply that each idempotent topological unosemigroup is ditopological.

\section{Uniformizable topological unosemigroups}\label{s3}

In this section we shall prove that the dicontinuity of a continuous left (right) unit operation on a topological semigroup automatically follows from
the right (left) uniformizability of the topological semigroup.

We define a topological semigroup $S$ to be {\em left-uniformizable} (resp. {\em right-uniformizable}) if the topology of $X$ is generated by
a uniformity $\U$ such that for every entourage $U\in\U$ there is an entourage $V\in U$ such that $x\cdot B(y,V)\subset B(xy,U)$
(resp. $B(x,V)\cdot y\subset B(xy,U)$ for all points $x,y\in S$. Here $B(x,V)=\{z\in S:(x,z)\in V\}$ stands for the $V$-ball
centered at the point $x\in S$.

Each topological group $G$ is left (resp. right) uniformizable by its left (resp. right) uniformity  generated by the base consisting of
the entourages $\{(x,y)\in G\times G:y\in xU\}$ (resp.  $\{(x,y)\in G\times G:y\in Ux\}$)
where $U=U^{-1}$ runs over symmetric neighborhoods of the idempotent in $G$.

\begin{theorem}\label{t3.1} Each continuous left unit operation $\lambda:S\to S$ on a right-uniformizable topological semigroup $S$ is dicontinuous.
Consequently, each right-uniformizable topological left unosemigroup $(S,\lambda)$ is a ditopological left unosemigroup.
\end{theorem}

\begin{proof} Let $\U$ be a uniformity on $S$ witnessing that the semigroup $S$ is right-uniformizable.
To show that the left unit operation $\lambda:S\to S$ is
dicontinuous, fix a point $x\in S$ and a neighborhood $O_x\subset
S$ of $x$. Since $\U$ generates the topology of $S$, the
neighborhood $O_x$ contains the ball $B(x,U)=\{y\in S:(x,y)\in
U\}$ for some entourage $U\in\U$. Find an entourage $V\in\U$ such
that $V\circ V\circ V\subset U$ where $V\circ V\circ V=\{(x,y)\in
S\times S:\exists u,v\in S\;\;(x,u),(u,v),(v,y)\in V\}$. Since $S$
is right-uniformizable by $\U$, for the entourage $V$ there is an
entourage $W\in\U$ such that $W\subset V$ and $B(s,W)\cdot
t\subset B(st,W)$ for all points $s,t\in S$.

We claim that the neighborhoods $U_x=B(x,W)$ and
$W_{\lambda(x)}=B(\lambda(x),W)\cap \lambda(S)$ witness that the left unit
operation $\lambda$ is dicontinuous $x$. Indeed, take any point
$y\in (W_{\lambda(x)}\leftthreetimes U_x)\cap
\lambda^{-1}(W_{\lambda(x)})$. Then  $\lambda(y)\in
W_{\lambda(x)}$ and $wy\in U_x$ for some $w\in W_{\lambda(x)}$. It
follows from $w,\lambda(y)\in W_{\lambda(y)}=B(\lambda(x),W)$ that
$$\{wy,y\}=\{wy,\lambda(y)y\}\subset B(\lambda(x),W)\cdot y\subset B(\lambda(x)y,V)$$ and hence
$(wy,y)\in V\circ V$, which implies
$$y\in B(wy,V\circ V)\subset B(U_x,V\circ V)=B(B(x,W),V\circ V)\subset B(x,V\circ V\circ V)\subset B(x,U)\subset O_x.$$
\end{proof}

By analogy we can prove

\begin{theorem}\label{t3.2}  Each continuous right unit operation $\lambda:S\to S$ on a left-uniformizable topological semigroup $S$ is dicontinuous.
Consequently, each left-uniformizable topological right unosemigroup $(S,\lambda)$ is a ditopological right unosemigroup.
\end{theorem}

Theorems~\ref{t3.1} and \ref{t3.2} imply:

\begin{theorem}\label{t4.3} A topological unosemigroup $(S,\lambda,\rho)$ is ditopological if the topological semigroup $S$ is both left-uniformizable and right-uniformizable.
\end{theorem}

Each compact Hausdorff space $X$ carries a unique uniformity generating its topology.
The uniform continuity of continuous maps defined on a compact Hausdorff space implies:

\begin{corollary}\label{c3.4} Each compact Hausdorff topological (left, right) unosemigroup is ditopological.
\end{corollary}

Since discrete topological semigroups are (left and right) uniformizable by the discrete uniformity,
Theorem~\ref{t4.3} implies:

\begin{corollary}\label{c3.5} Each discrete topological (left, right) unosemigroup is ditopological.
\end{corollary}

\section{Ditopological inverse semigroups}\label{s4}

Continuous left and right unit operations appear naturally in topological inverse semigroups.
Let us recall that a semigroup $S$ is called {\em inverse} if for each element $x\in S$ there is a unique element $x^{-1}\in S$
(called {\em the inverse to} $x$) such that $xx^{-1}x=x$ and $x^{-1}xx^{-1}=x^{-1}$.
 By a {\em topological inverse semigroup} we understand an inverse semigroup
$S$ endowed with a topology such that the semigroup operation $\cdot :S\times S\to S$ and the operation of inversion $(\,)^{-1}:S\to S$ are continuous.

Each topological inverse semigroup $S$ possesses the canonical
continuous left unit operation $$\lambda:S\to
S,\;\;\;\lambda:x\mapsto xx^{-1},$$ and the canonical continuous
right unit operation $$\rho:S\to S,\;\;\;\rho:x\mapsto
x^{-1}x,$$which turn $S$ into a topological unosemigroup.

Observe that the sets $\lambda(S)=\{xx^{-1}:x\in S\}$ and $\rho(S)=\{x^{-1}x:x\in S\}$ coincide with the set $E(S)=\{x\in S:xx=x\}$ of idempotents of $S$. It is well-known that the set $E(S)$ is a commutative subsemigroup of the inverse semigroup $S$.

A topological inverse semigroup $S$ is called a {\em ditopological inverse semigroup} if the topological unosemigroup $(S,\lambda,\rho)$ is ditopological.

The following proposition shows that the dicontinuity of the left unit operation on a topological inverse semigroup
is equivalent to the dicontinuity of the right unit operation.

\begin{proposition}\label{p4.2} Let $S$ be a topological inverse semigroup. The left unit operation $\lambda:S\to S$, $\lambda:x\mapsto xx^{-1}$,
is dicontinuous at a point $x\in S$ if and only if the right unit
operation $\rho:S\to S$, $\rho:x\mapsto x^{-1}x$, is dicontinuous
at the point $x^{-1}\in S$.
\end{proposition}

\begin{proof} Assume that the left unit operation $\lambda$ is dicontinuous at a point $x\in X$. To show that the right unit operation $\rho$
is dicontinuous at the point $x^{-1}$, fix a neighborhood
$O_{x^{-1}}\subset S$ of $x^{-1}$. Then
$O_{x^{-1}}^{-1}=\{y^{-1}:y\in O_{x^{-1}}\}$ is a neighborhood  of
the point $x$ in $S$. By the dicontinuity of the left unit
operation $\lambda$ on $S$, the points $x$ and
$\lambda(x)=xx^{-1}$ have neighborhoods $U_{x}\subset S$ and $W_{xx^{-1}}\subset E(S)$
such that $(W_{xx^{-1}}\leftthreetimes U_{x})\cap
\lambda^{-1}(W_{xx^{-1}})\subset O_{x^{-1}}^{-1}$. After the
inversion the latter inclusion turns into
$(U_{x}^{-1}\rightthreetimes W_{xx^{-1}}^{-1})\cap
\rho^{-1}(W_{xx^{-1}})\subset O_{x^{-1}}$. So,
$U_{x^{-1}}=U_{x}^{-1}$ and $W_{\rho(x^{-1})}=W_{xx^{-1}}^{-1}\cap
W_{xx^{-1}}$ are neighborhoods of the points $x^{-1}$ and
$\rho(x^{-1})=xx^{-1}$ witnessing the dicontinuity of the right
unit operation $\rho$ at $x^{-1}$.

By analogy we can prove that the dicontinuity of the right unit operation $\rho$ at $x^{-1}\in S$ implies the dicontinuity
of the left unit operation $\lambda$ at the point $x$.
\end{proof}

By a {\em topological semilattice} we understand a commutative idempotent topological semigroup $S$.
Each topological semilattice is a topological inverse semigroup with $x^{-1}=x$ for all $x\in S$.

\begin{theorem}\label{t4.2} The class of ditopological inverse semigroups includes all topological groups, all topological
semilattices, all compact Hausdorff topological inverse semigroups, and all discrete topological inverse semigroups.
\end{theorem}

\begin{proof} Let $G$ be a topological group. In this case $\lambda$ and $\rho$ are constant operations assigning to each $x\in G$
the unique idempotent $e$ of the group $G$, so $\lambda(G)=\rho(G)=\{e\}$ is a singleton. By Proposition~\ref{p4.2}, to prove that $G$
is a ditopological inverse semigroup,
it suffices to check that the left unit operation $\lambda$ is dicontinuous at each point $x\in G$. Given any neighborhood
$O_x\subset G$ of $x$, put $U_x=O_x$ and $W_{\lambda(x)}=\{e\}$ and observe that
$$(W_{\lambda(x)}\leftthreetimes U_x)\cap\lambda^{-1}(W_{\lambda(x)})=(\{e\}\leftthreetimes U_x)\cap G=U_x=O_x,$$which means that the left unit operation $\lambda$ is discontinuous at $x$.
So, the topological group $G$ is ditopological.
The same conclusion can be also derived from the
right-uniformizabity of topological groups and Theorem~\ref{t3.1}.
\vskip5pt

Each topological semilattice, being an idempotent topological unosemigroup, is a ditopological unosemigroup according to Propositions~\ref{p:LI} and \ref{p:RI}.
Corollaries~\ref{c3.4} and \ref{c3.5} imply that the class of ditopological
inverse semigroups contains all compact Hausdorff topological inverse
semigroups and all discrete topological inverse semigroups.
\end{proof}

Now we present a simple example of a locally compact commutative
topological inverse semigroup $S$ which is not ditopological.

\begin{example}\label{bool} There exists a commutative topological inverse semigroup $S$ such that
\begin{enumerate}
\item $S$ is countable, metrizable, and locally compact;
\item the idempotent semilattice $E(S)$ of $S$ is compact and each subgroup of $S$
has cardinality $\le 2$;
\item the left unit operation $\lambda:S\to S$, $\lambda:x\mapsto xx^{-1}$, is not dicontinuous,
which implies that the topological Clifford semigroup $S$ is not ditopological.
\end{enumerate}
\end{example}

\begin{proof} Let $H$ be a two-element group and  $T=\{0\}\cup\{\frac1n\}_{n\in\mathbb N}\subset\IR$
be a convergent sequence endowed with the semilattice operation
$$xy=\begin{cases}
x&\mbox{if $x=y$},\\
0&\mbox{otherwise}.
\end{cases}
$$Then the product $T\times H$ is a commutative inverse semigroup whose idempotent semilattice
coincides with the set $E=T\times\{e\}$ where $e$ is the idempotent of the group $H$.  Let $h$ be the non-identity element of the two-element group $H$.

Endow the inverse semigroup $S=T\times H$ with the topology $\tau$
which induces the  original (compact metrizable) topology on the
set $T\times\{e\}$ and the discrete topology on $T\times\{h\}$. It
is easy to see that the topology $\tau$ is metrizable, locally
compact, and turns $S$ into a topological inverse semigroup.

We claim that the left unit operation $\lambda:S\to S$, $\lambda:x\mapsto xx^{-1}$, is not dicontinuous at the point $x=(0,h)\in
T\times H$. Assuming the opposite, for the neighborhood
$O_x=\{x\}$ we would find neighborhoods $U_x\subset S$ of $x$ and
$W_{\lambda(x)}\subset\lambda(S)$ of the idempotent $\lambda(x)=(0,e)$ such that
$$(W_{\lambda(x)}\leftthreetimes U_x)\cap \lambda^{-1}(W_{\lambda(x)})\subset O_x=\{x\}.$$
By the definition
of the topology on $S$, the neighborhood $W_{\lambda(x)}$ contains a
point $(\frac1n,e)$ for some $n\in\IN$. It follows from $(0,e)\cdot (\frac1n,h)=(0,h)=x\in U_x$ and $\lambda\big((\frac1n,h)\big))=(\frac1n,e)\in W_{\lambda(x)}$ that
$$(\tfrac1n,h)\in (W_{\lambda(x)}\leftthreetimes U_x)\cap \lambda^{-1}(W_{\lambda(x)})\subset\{x\}=\{(0,h)\}$$ and hence $\frac1n=0$,  which is a desired contradiction.
\end{proof}

\section{Operations over ditopological unosemigroups}\label{s:Op}

As we know from Theorem~\ref{t4.2}, the class of ditopological unosemigroups contains all topological groups, all topological semilattices,
and all compact Hausdorff topological inverse semigroups. In
this section we shall show that this class is stable under many
natural operations over topological unosemigroups.

\subsection{Subunosemigroups of topological (left, right) unosemigroups}
Let $(S,\lambda)$ be a topological left unosemigroup and $X\subset S$ be a subsemigroup such that $\lambda(X)\subset X$. Then $\lambda|X:X\to X$ is a continuous left unit operation on $X$, which implies that $(X,\lambda|X)$ is a topological left unosemigroup. Such left unosemigroup will be called a {\em left subunosemigroup} of $(S,\lambda)$.

By analogy we can define {\em right subunosemigroups} of topological right unosemigroups and {\em subunosemi\-groups} of topological unosemigroups.

Since the dicontinuity of a left or right unit operation $u:S\to
S$ on a topological semigroup $S$ implies the dicontinuity of the
restriction $u|X$ to any subsemigroup $X\subset S$ with
$u(X)\subset X$, we get:

\begin{proposition} If a topological (left,right) unosemigroup $S$ is ditopological, then so is any its subunosemigroup $X\subset S$.
\end{proposition}

\subsection{Tychonoff products of topological (left, right) unosemigroups}
For any family of topological left unosemigroups $(S_\alpha,\lambda_\alpha)$, $\alpha\in A$,
the Tychonoff product $S=\prod_{\alpha\in A}S_\alpha$ carries a natural structure of a topological left
 unosemigroup endowed with the left unit operation $\lambda:S\to S$, $\lambda:(x_\alpha)_{\alpha\in A}\mapsto\big(\lambda_\alpha(x_\alpha)\big)_{\alpha\in A}$.
  By analogy we can define the operation of Tychonoff product of (right) unosemigroups.

The dicontinuity of left (right) unit operations on the topological semigroups $S_\alpha$, $\alpha\in A$, implies the
dicontinuity of the left (right) unit operation on their Tychonoff product. This proves our next simple:

\begin{proposition} The Tychonoff product of ditopological (left, right) unosemigroups is a ditopological (left, right) unosemigroup.
\end{proposition}

\subsection{The reduced product of topological (left, right) unosemigroups}
Let $X,Y$ be two topological semigroups and $I\subset X$ be a closed two-sided ideal in $X$.
By the {\em reduced product} $X\times_{I}Y$ of the semigroups $X$ and $Y$ over the ideal $I$ we mean
the set $I\cup ((X\setminus I)\times Y)$ endowed with the smallest topology
such that
\begin{itemize}
\item the map $(X\setminus I)\times Y\hookrightarrow X\times_{I}Y$
is a topological embedding,
\item the projection
$\pi:X\times_{I}Y\rightarrow X$ is continuous.
\end{itemize}
Here $$\pi(z)=\begin{cases}
z&\mbox{if $z\in I$,}\\
x&\mbox{if $z=(x,y)\in (X\setminus I)\times Y$,}
\end{cases}
$$for any $z\in X\times_I Y$.

The semigroup operation on $X\times_I Y$ is defined as a unique binary operation on $X\times_I Y$ such that
the projection $q:X\times Y\to X\times_I Y$ defined by
$$q(x,y)=\begin{cases}
x&\mbox{if $x\in I$};\\
(x,y)&\mbox{otherwise}
\end{cases}
$$
is a semigroup homomorphism.

If the semigroups $X$ and $Y$ carry continuous left unit operations $\lambda_X:X\to X$ and $\lambda_Y:Y\to Y$
such that $\lambda_X(I)\subset I$, then the reduced product $X\times_I Y$ carries a natural left unit operation $\lambda:X\times_I Y\to X\times_I Y$
defined by the formula:
$$\lambda(z)=\begin{cases}
\lambda_X(z)&\mbox{if $z\in I$}\\
(\lambda_X(x),\lambda_Y(y))&\mbox{if $z=(x,y)\in (X\setminus I)\times Y$}.
\end{cases}
$$This formula is well-defined since for every $x\in X\setminus I$ the equality $x=\lambda_X(x)\cdot x\notin I$ implies that $\lambda_X(x)\notin I$.

The semigroup $X\times_I Y$ endowed with the left unit operation $\lambda$ is a topological left unosemigroup called the {\em reduced product}
of the topological left unosemigroups $(X,\lambda_X)$ and $(Y,\lambda_Y)$.

\begin{theorem}\label{RPL} If the topological left unosemigroups $\mathbf X=(X,\lambda_X)$ and $\mathbf Y=(Y,\lambda_Y)$ are ditopological, then
so is their reduced product $\mathbf X\times_I\mathbf Y=(X\times_I Y,\lambda)$.
\end{theorem}

\begin{proof} Assume that the topological left unosemigroups
$\mathbf X=(X,\lambda_X)$ and $\mathbf Y=(Y,\lambda_Y)$ are
ditopological. To show that $X\times_I Y$ is a ditopological left
unosemigroup, fix a point $z\in X\times_{I}Y$ and a neighborhood
$O_{z}\subset X\times_{I}Y$ of $z$. We divide the proof into two
parts.

First assume that $z=(x,y)\in (X\setminus I)\times Y$. In this
case we can assume that $O_{z}=O_{x}\times O_{y}$ for some open
neighborhoods $O_{x}\subset X\setminus I$ and $O_{y}\subset Y$ of
the points $x\in X\setminus I$ and $y\in Y$, respectively. Since
$X$ and $Y$ are left ditopological unosemigroups, there are open
neighborhoods $U_{x}\subset X$, $W_{\lambda_X(x)}\subset\lambda_X(X)$ of the points $x$,
$\lambda_X(x)$ and $U_{y}\subset Y$, $W_{\lambda_Y(y)}\subset\lambda_Y(Y)$ of $y$, $\lambda_Y(y)$ such that $(W_{\lambda_X(x)}\leftthreetimes
U_x)\cap \lambda_X^{-1}(W_{\lambda_X(x)})\subset O_x$ and
$(W_{\lambda_Y(y)}\leftthreetimes
U_y)\cap\lambda_Y^{-1}(W_{\lambda_Y(y)})\subset O_y.$

It follows from $\lambda(x)\cdot x=x\notin I$ that $\lambda(x)\notin I$.
So, we can assume that the sets $U_x$ and $W_{\lambda(x)}$ are
contained in $X\setminus I$.

We claim that the open sets $U_{z}=U_{x}\times U_{y}$ and
$W_{\lambda(z)}=W_{\lambda_X(x)}\times W_{\lambda_Y(y)}$ witness that
the left unit operation $\lambda$ on $X\times_I Y$ is dicontinuous
at the point $z$. Given any point $u\in X\times_I Y$ such that $u
\in (W_{\lambda(z)}\leftthreetimes U_z)\cap
\lambda^{-1}(W_{\lambda(z)})$, we need to show that $u \in O_z$.
It follows that $\lambda(u)\in W_{\lambda(z)}$ and $wu\in
U_z$ for some $w\in W_{\lambda(z)}$, which implies that $w,u\notin
I$ and hence $w=(x_w,y_w)$ and $u=(x_u,y_u)$ for some points
$x_w,x_u\in X\setminus I$ and $y_w,y_u\in Y$. Since
$wu=(x_wx_u,y_wy_u)\in U_z=U_x\times U_y$ and
$\lambda(u)=(\lambda_X(x_u),\lambda_Y(y_u))\in
W_{\lambda(z)}=W_{\lambda_X(x)}\times W_{\lambda_Y(y)}$, we obtain
$x_wx_u\in U_x, y_wy_u\in U_y$ and $\lambda_X(x_u)\in
W_{\lambda_X(x)}, \lambda_Y(y_u)\in W_{\lambda_Y(y)}$. Hence
$x_u\in (W_{\lambda_X(x)}\leftthreetimes U_x)\cap
\lambda_X^{-1}(W_{\lambda_X(x)})\subset O_x$ and $y_u\in
(W_{\lambda_Y(y)}\leftthreetimes U_y)\cap
\lambda_Y^{-1}(W_{\lambda_Y(y)})\subset O_y$ and thus $u \in O_x\times O_y=O_z$.

In case $z\in I$, we can assume that $O_{z}$ is of the form
$O_{z}=\pi^{-1}(O'_{z})$, where $O'_{z}$ is an open neighborhood
of the point $z\in I$ in the left unosemigroup $X$. Since $\lambda_X$ is
left dicontinuous at $z$, there are open neighborhoods $U'_{z}\subset X$ and
$W'_{\lambda_X(z)}\subset\lambda_X(X)$ of the points $z\in X$ and $\lambda_X(z)\in\lambda_X(X)$ such that
$(W'_{\lambda_X(z)}\leftthreetimes U'_{z})\cap
\lambda_X^{-1}(W'_{\lambda_X(z)})\subset O'_z$. The latter inclusion implies that for the open neighborhoods
$U_{z}=\pi^{-1}(U'_{z})\subset X\times_{I}Y$ of $z$ and
$W_{\lambda(z)}=\lambda(X\times_I Y)\cap \pi^{-1}(W'_{\lambda_X(z)})\subset \lambda(X\times_{I}Y)$ of
$\lambda(z)$ we get
 $$(W_{\lambda(x)}\leftthreetimes U_z)\cap\lambda^{-1}(W_{\lambda(z)})\subset
\pi^{-1}\big((W'_{\lambda_X(z)}\leftthreetimes U'_{z})\cap
\lambda_X^{-1}(W'_{\lambda_X(z)})\big)\subset  \pi^{-1}(O'_z)=O_z,$$
which witnesses that the left unit operation $\lambda$
on $X\times_I Y$ is left dicontinuous at $z$.
 Thus, the
reduced product $\mathbf X\times_I\mathbf
Y=(X\times_{I}Y,\lambda)$ is a ditopological left unosemigroup.
\end{proof}

\smallskip

By analogy we can introduce the reduced product of topological
right unosemigroups. Namely, if $\mathbf X=(X,\rho_X)$ and $\mathbf
Y=(Y,\rho_Y)$ are two topological right unosemigroups and $I\subset X$ is
a closed two-sided ideal with $\rho_X(I)\subset I$, then the
reduced product $X\times_I Y$ carries an induced right unit
operation $\rho:X\times_I Y\to X\times_I Y$ defined by
$$\rho(z)=\begin{cases}
\rho_X(z)&\mbox{if $z\in I$}\\
(\rho_X(x),\rho_Y(y))&\mbox{if $z=(x,y)\in (X\setminus I)\times Y$}.
\end{cases}
$$
The reduced product $X\times_I Y$ endowed with the right unit operation $\rho$ is a topological right
unosemigroup called the {\em reduced product} of the topological right unosemigroups $(X,\lambda_X)$ and $(Y,\lambda_Y)$.

By analogy with Theorem~\ref{RPL}, we can prove:
\begin{theorem}\label{RPR} If the topological right unosemigroups $\mathbf X=(X,\rho_X)$ and $\mathbf Y=(Y,\rho_Y)$ are ditopological,
then so is their reduced product $\mathbf X\times_I \mathbf Y=(X\times_I Y,\rho)$.
\end{theorem}

The above discussion implies that for topological unosemigroups $\mathbf X=(X,\lambda_X,\rho_X)$ and $\mathbf Y=(Y,\lambda_Y,\rho_Y)$
and a closed two-sided ideal $I\subset X$ with $\lambda_X(I)\cup\rho_X(I)\subset I$,
the triple $\mathbf X\times_I\mathbf Y=(X\times_I Y,\lambda,\rho)$ is a topological unosemigroup. This topological unosemigroups  will be called
{\em reduced product} of the topological unosemigroups $\mathbf X$ and $\mathbf Y$. Theorems~\ref{RPL} and \ref{RPR} imply:

\begin{corollary}\label{RP} If topological unosemigroups $\mathbf X$ and $\mathbf Y$ are ditopological,
then so is their reduced product $\mathbf X\times_I \mathbf Y$.
\end{corollary}

Now we present some important examples of reduced products.

\begin{example} Let $G$ be a topological group and let
$\mathbf 2=(\{0,1\},\min)$ be a two-element semilattice endowed
with the discrete topology. By Theorem~\ref{t4.2}, the semigroups
$G$ and $\mathbf 2$ endowed with the canonical left and right unit
operations are ditopological inverse semigroups and by
Corollary~\ref{RP}, so is their reduced product $\dot G=\mathbf
2\times_{\{0\}}G$ called the {\em $0$-extension} of $G$.
\end{example}

\begin{example} Let $G$ be a topological group and let $\II$ be the unit interval $[0,1]$ endowed with the semilattice operation of minimum.
By Theorem~\ref{t4.2}, the semigroups $G$ and $\II$ endowed with
the canonical left and right unit operations are ditopological
inverse semigroups and by Corollary~\ref{RP}, so is their reduced
product $\hat G=\II \times_{\{0\}}G$ called {\em the cone over
$G$}.
\end{example}

The $0$-extensions and cones of topological groups will be
essentially used in the paper \cite{BP} devoted to constructing
embeddings of Clifford ditopological inverse semigroups into
Tychonoff products of topological semilattices and cones over
topological groups.

\subsection{Semidirect products of topological unosemigroups}
In this subsection we shall consider the operation of a semidirect product of topological (left, right) unosemigroups.
Let us mention that semidirect products of semigroups were studied in Chapter 2 of \cite{CHK2}.

By a {\em continuous action} of a topological semigroup $F$
on a topological semigroup $S$ we understand a continuous
function $\alpha:F\times S\to S$ having the following two properties:
\begin{itemize}
\item for each $f\in F$ the function $\alpha_f:S\to S$,
$\alpha_f:x\mapsto\alpha(f,x)$, is a semigroup homomorphism
of $S$; \item $\alpha_{fg}=\alpha_f\circ\alpha_g$ for each $f,g\in
F$.
\end{itemize}

The action $\alpha:F\times S\to S$ induces a continuous associative binary operation
$$(s,f)\cdot(t,g)=(s\cdot \alpha_f(t),f\cdot g)$$on the product $S\times F$.
The product $S\times F$ endowed with this binary operation is
denoted by $S\times_\alpha F$ and called the {\em semidirect product} of the topological semigroups
$S$ and $F$.

We shall say that the action $\alpha:F\times S\to S$
\begin{itemize}
\item respects a (left, right) unit operation $u:F\to F$ on $F$ if
$\alpha(u(f),s)=s$ for all $(f,s)\in F\times S$;
\item respects a (left, right) unit operation $u:S\to S$ on $S$ if
$\alpha(f,u(s))=u(s)$ for all $(f,s)\in F\times S$.
\end{itemize}

If $(F,\lambda_F)$ and $(S,\lambda_S)$ are topological left unosemigroups
and a continuous action $\alpha:F\times S\to S$ of $F$ on $S$
respects the left unit operation $\lambda_F$, then the unary
operation $$\lambda:S\times_\alpha F\to
S\times_\alpha F,\;\;\;
\lambda:(s,f)\mapsto(\lambda_S(s),\lambda_F(f)),$$
is a continuous left unit operation on the semidirect product
$S\times_\alpha F$ as
$$(\lambda_S(s),\lambda_F(f))\cdot(s,f)=(\lambda_S(s)\cdot\alpha(\lambda_F(f),s),\lambda_F(f)\cdot f)=
(\lambda_S(s)\cdot s,f)=(s,f)$$for all $(s,f)\in S\times F$.

Therefore, $\mathbf S\times_\alpha\mathbf F=(S\times_\alpha F,\lambda)$ is a topological left unosemigroup,
called the {\em semidirect product} of the topological left unosemigroups $\mathbf S=(S,\lambda_S)$ and $\mathbf F=(F,\lambda_F)$.

\begin{theorem}\label{SDL} Let $\mathbf S=(S,\lambda_S)$ and $\mathbf F=(F,\lambda_F)$ be topological left unosemigroups and $\alpha:F\times S\to F$ be a continuous action of $F$ on $S$, which respects the left unit operation $\lambda_F$ of the left unosemigroup $F$. The topological left unosemigroup $\mathbf S\times_\alpha \mathbf F$ is ditopological if and only if the topological left unosemigroups $\mathbf S$ and $\mathbf F$ are ditopological.
\end{theorem}

\begin{proof} To prove the ``only if'' part, assume that $\mathbf S\times_\alpha \mathbf F$ is a
ditopological left unosemigroup. We need to prove that the topological left unosemigroups $\mathbf S$ and $\mathbf F$ are ditopological.

To prove the dicontinuity of the left unit operation $\lambda_S$, fix any point $s\in S$ and a neighborhood $O_s$ of $s$ in $S$. Fix any point $f\in
F$ and consider the neighborhood $O_{(s,f)}=O_s\times F$ of $(s,f)$ in the topological semigroup $S\times_\alpha F$.
The dicontinuity of the left unit operation
$\lambda$ on $S\times_\alpha F$ yields
neighborhoods $U_{(s,f)}\subset S\times_\alpha F$ and $W_{\lambda(s,f)}\subset \lambda(S\times_\alpha F)=\lambda_S(S)\times\lambda_F(F)$ of $(s,f)$ and
$\lambda(s,f)=(\lambda_S(s),\lambda_F(f))$ such that
$(W_{\lambda(s,f)}\leftthreetimes U_{(s,f)})\cap
\lambda^{-1}(W_{\lambda(s,f)})\subset O_{(s,f)}$. Clearly, we can
assume that these neighborhoods are of the form
$U_{(s,f)}=U_s\times U_f$ and
$W_{\lambda(s,f)}=W_{\lambda_S(s)}\times W_{\lambda_F(f)}$ for
some open sets $U_s\subset S$, $W_{\lambda_S(s)}\subset \lambda_S(S)$,
$U_f\subset F$, and $W_{\lambda_F(f)}\subset \lambda_F(F)$.

We claim that the neighborhoods $U_s$ and $W_{\lambda_S(s)}$
witness that  $\lambda_S$ is dicontinuous at $s$. Let $t\in
(W_{\lambda_S(s)}\leftthreetimes U_s)\cap
\lambda^{-1}_S(W_{\lambda_S(s)})$. This implies $\lambda_S(t)\in
W_{\lambda_S(s)}$ and $wt\in U_s$ for some $w\in
W_{\lambda_S(s)}$. Then for the elements $(t,f)\in S\times_\alpha F$
and $(w,\lambda_F(f))\in W_{\lambda(s,f)}$ we get
$$(w,\lambda_F(f))\cdot (t,f)=(w\cdot \alpha_{\lambda_F(f)}(t),\lambda_F(f)\cdot f)=(wt,f)\in U_s\times
U_f=U_{(s,f)}$$ and $$\lambda(s,f)=(\lambda_S(s),\lambda_F(f))\in
W_{\lambda_S(s)}\times W_{\lambda_F(f)}=W_{\lambda(s,f)}.$$ Here we used the fact that the action $\alpha$ respects
the left unit operation $\lambda_F$. The choice of the neighborhoods $U_{(s,f)}$ and $W_{\lambda(s,f)}$ guarantees
that $(t,f)\in O_{(s,f)}$ and hence $t \in O_s$.
\smallskip

To check the dicontinuity of the left unit operation
$\lambda_F:F\rightarrow F$, take any point $f \in F$ and a
neighborhood $O_f\subset F$ of $f$ in $F$. Fix any element $s\in S$ and
consider the neighborhood $O_{(s,f)}=S\times O_f$ of $(s,f)$ in $S\times_\alpha F$.
The dicontinuity of the left unit operation $\lambda$ on $\mathbf S\times_\alpha
\mathbf F$ yields neighborhoods
$U_{(s,f)}\subset S\times_\alpha F$ and $W_{\lambda(s,f)}\subset\lambda(S\times_\alpha F)$ of $(s,f)$ and
$\lambda(s,f)=(\lambda_S(s),\lambda_F(f))$ such that $(W_{\lambda(s,f)}\leftthreetimes U_{(s,f)}) \cap
\lambda^{-1}(W_{\lambda(s,f)})\subset O_{(s,f)}$. We lose no
generality assuming that $U_{(s,f)}=U_s\times U_f$ and
$W_{\lambda(s,f)}=W_{\lambda_S(s)}\times W_{\lambda_F(f)}$ for
some open sets $U_s\subset S$, $W_{\lambda_S(s)}\subset \lambda_S(S)$,
$U_f\subset F$, and $W_{\lambda_F(f)}\subset \lambda_F(F)$.

We claim that the neighborhoods $U_f$ and $W_{\lambda_F(f)}$
witness the dicontinuity of $\lambda_F$ at the point $f$. Given
any point $g\in (W_{\lambda_F(f)}\leftthreetimes U_f)\cap
\lambda^{-1}_F(W_{\lambda_F(f)})$, observe that $\lambda_F(g)\in
W_{\lambda_F(f)}$ and $wg\in U_f$ for some $w\in
W_{\lambda_F(f)}\subset\lambda_F(F)$. Taking into account that the
action $\alpha$ respects the left unit operation $\lambda_F$ and
$w\in\lambda_F(F)$, we conclude that $\alpha_w(s)=s$. Then for the
elements $(s,g)$ and $(\lambda_S(s),w)\in W_{\lambda_S(s)}\times
W_{\lambda_F(f)}=W_{\lambda(s,f)}$ we get
$$(\lambda_S(s),w)\cdot (s,g)=(\lambda_S(s)\cdot \alpha_w(s),wg)=(\lambda_S(s)\cdot s,wg)=(s,g)\in U_{(s,f)}$$
and $\lambda(s,g)=(\lambda_S(s),\lambda_F(g))\in W_{\lambda_S(s)}\times W_{\lambda_F(f)}=W_{\lambda(s,f)}$, which means that
$$(s,g)\in (W_{\lambda(s,f)}\leftthreetimes U_{(s,f)})\cap\lambda^{-1}(W_{\lambda(s,f)})\subset O_{(s,f)}=S\times O_f$$
and hence $g\in O_f$. This completes the proof of the ``only if'' part of the theorem.
\smallskip

To prove the ``if'' part, assume that the topological left unosemigroups $\mathbf S$ and $\mathbf F$ are ditopological.
We need to check that the left unit operation $\lambda:(s,f)\mapsto(\lambda_S(s),\lambda_F(f))$ on $S\times_\alpha F$
is dicontinuous at every point $(s,f)\in S\times_\alpha F$. Fix any open neighborhood
$O_{(s,f)}$  of $(s,f)$ in $S\times_\alpha F$. We lose no generality
assuming that it is of basic form:
$O_{(s,f)}=O_s\times O_f$ where $O_s$ and $O_f$ are open
neighborhoods of $s$ and $f$ in $S$ and $F$, respectively.

By the
dicontinuity of the left unit operation $\lambda_F$ at $f$, there are
neighborhoods $U_f\subset F$ and $W_{\lambda_F(f)}\subset\lambda_F(F)$ of $f$ and
$\lambda_F(f)$ such that $(W_{\lambda_F(f)}\leftthreetimes
U_f )\cap \lambda_F^{-1}(W_{\lambda_F(f)})\subset O_f$.
By the
dicontinuity of the left unit operation $\lambda_S$ at $s$, there are
neighborhoods $U_s\subset S$ and $W_{\lambda_S(s)}\subset\lambda_S(S)$ of $s$ and
$\lambda_S(s)$ such that $(W_{\lambda_S(s)}\leftthreetimes
U_s)\cap \lambda_S^{-1}(W_{\lambda_S(s)})\subset O_s$.

We claim that the neighborhoods $U_{(s,f)}=U_s\times U_f$ and
$W_{(s,f)}=W_{\lambda_S(s)}\times W_{\lambda_F(f)}$ of $(s,f)$ and $\lambda(s,f)$ witness that
the left unit operation $\lambda$ is dicontinuous at $(s,f)$.
Given any pair $(t,g)\in (W_{\lambda(s,f)}\leftthreetimes U_{(s,f)})\cap\lambda^{-1}(W_{\lambda(s,f)})$,
we need to show that $(t,g)\in O_{(s,f)}$. It follows that $(w,h)\cdot (t,g)\in U_{(s,g)}$ for some pair $(w,h)\in W_{\lambda(s,f)}$.
Taking into account that the action $\alpha$ respects the left unit operation $\lambda_F$ and $h\in W_{\lambda_F(f)}\subset\lambda_F(F)$,
we conclude that
$(w,h)\cdot(t,g)=(w\cdot\alpha_h(t),hg)=(w t,hg)$ and hence $(wt,hg)\in U_{(s,f)}=U_s\times U_f$ and
$$(t,g)\in \big((W_{\lambda_S(s)}\leftthreetimes U_s)\cap\lambda_S^{-1}(W_{\lambda_S(s)})\big)\times
\big((W_{\lambda_F(f)}\leftthreetimes U_f)\cap\lambda_F^{-1}(W_{\lambda_F(f)})\big)\subset O_s\times O_f=O_{(s,f)}.$$
\end{proof}

Now, given two topological right unosemigroups $\mathbf S=(S,\rho_S)$ and $\mathbf F=(F,\rho_F)$ and
a continuous action $\alpha:F\times S\to S$ of $F$ on $S$, we shall define a right unit operation on the semidirect product $S\times_\alpha F$.
This can be done under an additional assumption that the action $\alpha$ is {\em $\rho_S$-invertible}
in the sense that for every $f\in F$ the restriction $\bar\alpha_f=\alpha_f|\rho_S(S)$ is a bijective map of $\rho_S(S)$ and the map
$$\alpha^{-}:F\times \rho_S(S)\to\rho_S(S),\;\;\alpha^-:(f,s)\mapsto \bar\alpha^{-1}_f(s),$$
is continuous.

In this case the map $\rho:S\times_\alpha F\to S\times_\alpha F$ defined by $$\rho(s,f)=(\bar\alpha^{-1}_f(\rho_S(s)),\rho_F(f))=(\alpha^-(f,\rho_S(s)),\rho_F(f))$$ is continuous.

Since
$$
\begin{aligned}
(s,f)\cdot \rho(s,f)&=(s,f)\cdot(\bar \alpha^{-1}_f(\rho_S(s)),\rho_F(f))=(s\cdot\alpha_f\circ\bar\alpha^{-1}_f(\rho_S(s)),f\cdot \rho_F(f))=\\
&=(s\cdot\bar\alpha_f\circ\bar\alpha^{-1}_f(\rho_S(s)),f)=(s\cdot\rho_S(s),f)=(s,f),
\end{aligned}
$$
the map $\rho$ is a continuous right unit operation on $S\times_\alpha F$. Therefore, $\mathbf S\times_\alpha \mathbf F=(S\times_\alpha F,\rho)$
is a topological right unosemigroup, called the {\em semidirect product} of the topological right unosemigroups $\mathbf S=(S,\rho_S)$ and $\mathbf F=(F,\rho_F)$.

Let us observe that if the action $\alpha$ respects the right unit operation $\rho_S$, then for every $f\in F$
the restriction $\bar \alpha_f=\alpha_f|\rho_S(S)$ is an identity map of $\rho_S(S)$ and hence the action $\alpha$ is $\rho_S$-invertible.
Moreover, in this case $\rho(s,f)=(\rho_S(s),\rho_F(f))$ for all $(s,f)\in S\times F$.

The following propositions will help us to detect $\rho_S$-invertible actions.

\begin{proposition} A continuous action $\alpha:F\times S\to S$ of a topological right unosemigroup $(F,\rho_F)$ on a
topological right unosemigroup $(S,\rho_S)$ is $\rho_S$-invertible if
\begin{enumerate}
\item $\alpha$ respects the right unit operation $\rho_F$;
\item $\alpha_f(\rho_S(S))=\rho_S(S)$ for all $f\in F$;
\item there is a continuous unary operation $(\,)^{-1}:F\to F$ such that $\rho_F(f)=f^{-1}f$ for all $f\in F$.
\end{enumerate}
\end{proposition}

\begin{proof} Taking into account that the action $\alpha$ preserves the right unit operation $\rho_F:F\to F$, $\rho_F:f\mapsto f^{-1}f$,
we conclude that for every $f\in F$ and $s\in S$ we get $$s=\alpha(\rho_F(f),s)=\alpha(f^{-1}f,s)=\alpha_{f^{-1}f}(s)=\alpha_{f^{-1}}\circ\alpha_f(s),$$
which implies that the homomorphism $\alpha_f:S\to S$ is injective and hence has the inverse $\alpha_f^{-1}:\alpha_f(S)\to S$.
It follows from $\alpha_f(\rho_S(S))=\rho_S(S)$ that the restriction $\bar\alpha_f=\alpha_f|\rho_S(S)$ is a bijective map of $\rho_S(S)$.

It remains to check that the map $\alpha^{-}:F\times\rho_S(S)\to \rho_S(S)$, $\alpha^{-}:(f,s)\mapsto \bar \alpha^{-1}_f(s)$, is continuous.
For this observe that the function $\alpha^{-}$ coincides with the  continuous function
$\beta:F\times \rho_S(S)\to \rho_S(S)$, $\beta(f,s)\mapsto\alpha(f^{-1},s)$. Indeed, given any $f\in F$ and $s\in\rho_S(S)$,
we can find a unique point $x\in \rho_S(S)$ with $s=\bar\alpha_f(x)=\alpha_f(x)$ and conclude that
$$\beta(f,s)=\alpha(f^{-1},s)=\alpha_{f^{-1}}(s)=\alpha_{f^{-1}}\circ\alpha_f(x)=\alpha_{f^{-1}f}(x)=
\alpha_{\rho_F(f)}(x)=x=\bar\alpha_f^{-1}(s)=\alpha^{-}(f,s).$$
\end{proof}

Now we study the semidirect products of ditopological right unosemigroups.

\begin{theorem}\label{SDR} Let $\mathbf S=(S,\rho_S)$ and $\mathbf F=(F,\rho_F)$ be topological right unosemigroups
and  $\alpha:F\times S\to S$ be a $\rho_S$-invertible continuous action of $F$ on $S$. The semidirect product $\mathbf S\times_\alpha\mathbf F$ is
a ditopological right unosemigroup if and only if  the topological right unosemigroups  $\mathbf S$ and $\mathbf F$ are ditopological.
\end{theorem}

\begin{proof} To prove the ``only if'' part, assume that the topological right unosemigroup $\mathbf S\times_\alpha \mathbf F=(S\times_\alpha F,\rho)$
is ditopological. We need to show that the right unit operations $\rho_S$ and $\rho_F$ are dicontinuous.
\smallskip

To prove the dicontinuity of the right unit operation $\rho_S$, fix any point $s\in S$ and a neighborhood $O_s\subset S$ of $s$. Fix any point $f\in\rho_F(F)\subset F$ and
consider the homomorphism $\alpha_f:S\to S$, whose restriction $\bar\alpha_f$ is a bijective map of the
set $\rho_S(S)$. We claim that $\bar\alpha_f$ is the identity map of $\rho_S(S)$. It follows from $f\in\rho_F(F)$ that $f=\rho_F(g)$ for some $g\in F$. The equality $g=g\cdot\rho_F(g)=gf$ implies that $\bar\alpha_g=\bar\alpha_g\circ\bar \alpha_f$, which is possible only in case of identity map $\bar\alpha_f$.

Now consider the point $(s,f)\in S\times_\alpha F$ and its neighborhood $O_{(s,f)}=O_s\times F$.
It follows that $\rho(s,f)=(\bar\alpha_f^{-1}(\rho_S(s)),\rho_S(f))=(\rho_S(s),\rho_F(f))$.
 The dicontinuity of the right unit operation $\rho$ on $S\times_\alpha F$ yields neighborhoods $U_{(s,f)}\subset S\times_\alpha F$
 and $W_{\rho(s,f)}\subset \rho(S\times_\alpha F)=\rho_S(S)\times\rho_F(F)$ of $(s,f)$ and $\rho(s,f)=(\rho_S(s),\rho_F(f))$ such that
$(U_{(s,f)}\rightthreetimes W_{\rho(s,f)})\cap\rho^{-1}(W_{\rho(s,f)})\subset O_{(s,f)}$.
 We lose no
generality assuming that $U_{(s,f)}=U_s\times U_f$ and
$W_{\rho(s,f)}=W_{\rho_S(s)}\times W_{\rho_F(f)}$ for
some open sets $U_s\subset S$, $W_{\rho_S(s)}\subset \rho_S(S)$,
$U_f\subset F$, and $W_{\rho_F(f)}\subset \rho_F(F)$.

We claim that the neighborhoods $U_s$ and $W_{\rho_S(s)}$ have the
required property: $(U_s\rightthreetimes
W_{\rho_S(s)})\cap\rho_S^{-1}(W_{\rho_S(s)})\subset O_s$. Given
any point $t\in (U_s\rightthreetimes
W_{\rho_S(s)})\cap\rho_S^{-1}(W_{\rho_S(s)})$, find a point $w\in
W_{\rho_S(s)}\subset\rho_S(S)$ such that $tw\in U_s$. Consider the
point $(t,f)\in S\times_\alpha F$ and $(w,\rho_F(f))\in
W_{\rho_S(s)}\times W_{\rho_F(f)}=W_{\rho(s,f)}$ and observe that
$$(t,f)\cdot (w,\rho_F(f))=(t\cdot\bar\alpha_f(w),f\cdot\rho_S(f))=(tw,f)\in U_s\times U_f.$$
Since $$\rho(t,f)=(\bar\alpha_f^{-1}(\rho_S(t)),\rho_F(f))=(\rho_S(t),\rho_F(f))\in W_{\rho_S(s)}\times W_{\rho_F(f)}=W_{\rho(s,f)},$$ we get the desired inclusion
$$(t,f)\in (U_{(s,f)} \rightthreetimes W_{\rho(s,f)})\cap\rho^{-1}(W_{\rho(s,f)})\subset O_{(s,f)}=O_s\times F,$$
which implies $t\in O_s$.
\smallskip

Next, we show that the right unit operation $\rho_F$ on $F$ is
dicontinuous at every point $f\in F$. Fix any neighborhood $O_f$
of $f$ in $F$. Fix any point $s\in S$ and consider the pair
$(s,f)$ and its neighborhood $O_{(s,f)}=S\times O_f$ in
$S\times_\alpha F$. The dicontinuity of the right unit operation
$\rho$ on $S\times_\alpha F$ yields neighborhoods
$U_{(s,f)}\subset S\times_\alpha F$ and
$W_{\rho(s,f)}\subset\rho(S\times_\alpha
F)=\rho_S(S)\times\rho_F(F)$ of the elements $(s,f)$ and
$\rho(s,f)$ such that $(U_{(s,f)} \rightthreetimes
W_{\rho(s,f)})\cap \rho^{-1}(W_{\rho(s,f)})\subset O_{(s,f)}$.
Consider the point $r=\bar\alpha^{-1}_f(\rho_S(s))$ and observe
that $\rho(s,f)=(r,\rho_F(f))$. Without lose of generality we can
assume that $W_{\rho(s,f)}=W_{r}\times W_{\rho_F(f)}$ and
$U_{(s,f)}=U_s\times U_f$ for some open sets $W_r\subset
\rho_S(S)$, $W_{\rho_F(f)}\subset \rho_F(F)$, $U_s\subset S$ and
$U_f\subset F$.

Consider the continuous functions $\beta:F\to\rho_S(S)$, $\beta:g\mapsto\bar\alpha_g^{-1}(\rho_S(s))=\alpha^-(g,\rho_S(s))$, and $\gamma:F\to S$, $\gamma:g\mapsto s\cdot\alpha(g,r)$, and observe that $\beta(f)=r$ and $$\gamma(f)=s\cdot \alpha_f(r)=s\cdot\bar\alpha_f\circ\bar\alpha^{-1}_f(\rho_S(s))=s\cdot\rho_S(s)=s.$$
Using the continuity of the functions $\beta$ and $\gamma$, find a neighborhood $U'_f\subset U_f$ of $f$ such that $\beta(U'_f)\subset W_r$
and $\gamma(U'_f)\subset U_s$.

We claim that the neighborhoods $U'_f$ and $W_{\rho_F(f)}$ have the required property:
$(U'_f\rightthreetimes W_{\rho_F(f)})\cap\rho_F^{-1}(W_{\rho_F(f)})\subset O_f$.
Given any point $g\in (U'_f\rightthreetimes W_{\rho_F(f)})\cap\rho_F^{-1}(W_{\rho_F(f)})$, find a point $h\in W_{\rho_F(f)}\subset\rho_F(F)$ with $gh\in U'_f$. Consider the points $(s,g)\in S\times_\alpha F$ and $(r,h)\in W_r\times W_{\rho_F(f)}=W_{\rho(s,f)}$. Since $h\in\rho_F(F)$, the map $\bar\alpha_h$ is an identity homeomorphism of $\rho_S(S)$. Hence, $\alpha_h(r)=r$ and
$$s\cdot \alpha_g(r)=s\cdot\alpha_g(\alpha_h(r))=s\cdot\alpha_{gh}(r)=\gamma(gh)\in\gamma(U'_f)\subset U_s.$$

Also $\bar\alpha_{gh}=\bar\alpha_g\circ\bar\alpha_h=\bar\alpha_g$ implies that$$\rho(s,g)=(\bar\alpha_g^{-1}(\rho_S(s)),\rho_F(g))=(\bar\alpha^{-1}_{gh}(\rho_S(s)),\rho_F(g))=
(\beta(gh),\rho_F(g))\in \beta(U'_f)\times W_{\rho_F(f)}\subset W_r\times W_{\rho_F(f)}=W_{\rho(s,f)}.$$ Since
$$(s,g)\cdot (r,h)=(s\cdot \alpha_g(r),gh)\in U_s\times U'_f\subset U_{(s,f)},$$
we get $$(s,g)\in (U_{(s,f)} \rightthreetimes W_{\rho(s,f)})\cap\rho^{-1}(W_{\rho(s,f)})\subset O_{(s,f)}=S\times O_f$$and hence $g\in O_f$, which completes the proof of the ``only if'' part of the theorem.
\smallskip

To prove the ``if'' part, assume that the topological right unosemigroups $(S,\rho_S)$ and $(F,\rho_F)$ are ditopological. We need to check that the right unit
operation $\rho:S\times_\alpha F \rightarrow
S\times_\alpha F$, $\rho:(s,f)\mapsto (\bar\alpha_f^{-1}(\rho_S(s)),\rho_F(f))$, is
dicontinuous at each point $(s,f)\in S\times_\alpha F$.

Fix any neighborhood $O_{(s,f)}$ of the point $(s,f)$ in the
topological semigroup $S\times_\alpha F$. We lose no generality
assuming that this neighborhood is of basic form:
$O_{(s,f)}=O_s\times O_f$ where $O_s$ and $O_f$ are open
neighborhoods of $s$ and $f$ in $S$ and $F$, respectively. The
dicontinuity of the right unit operation $\rho_S$ yields
neighborhoods $U_s\subset S$ and $W_{\rho_S(s)}\subset \rho_S(S)$ of $s$ and $\rho_S(s)$  such that $(U_s \rightthreetimes W_{\rho_S(s)})\cap
\rho_S^{-1}(W_{\rho_S(s)})\subset O_s$.

Now consider the point $r=\bar\alpha_{f}^{-1}(\rho_S(s))\in \rho_S(S)$ and
observe that
$$\alpha(f,r)=\bar\alpha_f(r)=\bar\alpha_f\circ \bar\alpha_{f}^{-1}(\rho_S(s))=\rho_S(s)\in
W_{\rho_S(s)}.$$ The continuity of the action
$\alpha|(F\times\rho_S(S))$ yields a neighborhood $O'_f\subset
O_f$ of $f$ and a neighborhood $W_{r}\subset \rho_S(S)$ of $r$
such that $\alpha(O'_f\times W_{r})\subset W_{\rho_S(s)}$.

Since the right unit operation $\rho_F$ is dicontinuous at $f$, there
are neighborhoods $U_f\subset F$ and $W_{\rho_F(f)}\subset \rho_F(F)$ of $f$
and $\rho_F(f)$ such that  $(U_f \rightthreetimes W_{\rho_F(f)})\cap \rho_F^{-1}(W_{\rho_F(f)})\subset O'_f$.

The $\rho_S$-invertibility of the action $\alpha$ guarantees that $\rho(S\times_\alpha F)=\{(\bar\alpha_f^{-1}(\rho_S(s)),\rho_F(f)):(s,f)\in S\times_\alpha F\}=\rho_S(S)\times\rho_F(F)$, which implies that the product $W_{\rho(s,f)}=W_r\times W_{\rho_F(f)}$ is a neighborhood of the point $\rho(s,f)=(\bar\alpha_f^{-1}(\rho_S(s)),\rho_F(f))=(r,\rho_F(f))$ in $\rho(S\times_\alpha F)=\rho_S(S)\times\rho_F(F)$. We claim that the
neighborhoods $U_{(s,f)}=U_s\times U_f$ and $W_{\rho(s,f)}=W_r\times W_{\rho_F(f)}$ have the required property: $(U_{(s,f)}\rightthreetimes W_{\rho(s,f)})\cap\rho^{-1}(W_{\rho(s,f)})\subset O_{(s,f)}$.
Fix any point $(t,g)\in (U_{(s,f)} \rightthreetimes W_{\rho(s,f)})\cap\rho^{-1}(W_{\rho(s,f)})$. It follows that $(t,g)(w,h)\in U_s\times U_f$ for some $(w,h)\in W_{\rho(s,f)}$.

Observe that
$$W_{r}\times W_{\rho_F(f)}=W_{\rho(s,f)}\ni\rho(t,g)=
(\bar\alpha^{-1}_g(\rho_S(t)),\rho_F(g))$$implies $\rho_F(g)\in
W_{\rho_F(f)}$ and $\bar\alpha^{-1}_g(\rho_S(t))\in W_{r}$. Also
$(w,h)\in W_{\rho(s,f)}=W_{r}\times W_{\rho_F(f)}$ yields
$w\in W_{r}$ and $h\in W_{\rho_F(f)}$.

It follows from $\rho_F(g),h\in W_{\rho_F(f)}$ and $gh\in U_f$ that
$g\in O'_f\subset O_f$. Then
$$\rho_S(t)=\bar\alpha_g\circ \bar\alpha_g^{-1}(\rho_S(t))=\alpha(g,\bar\alpha_{g}^{-1}(\rho_S(t)))\in
\alpha(O'_f\times W_r)\subset W_{\rho_S(s)}.$$ By the same reason,
$$\alpha_g(w)=\alpha(g,w)\in\alpha(O_f'\times W_r)\subset W_{\rho_S(s)}.$$
Observe also that
$$(t\cdot \alpha_g(w),gh)=(t,g)\cdot(w,h)\in U_{(s,t)}=U_s\times U_f$$ yields $gh\in U_f$ and $t\cdot \alpha_g(w)\in U_s$. The choice of the neighborhoods $U_s$ and $W_{\rho_S(s)}\supset\{\rho_S(t),\alpha_g(w)\}$ guarantees that $t\in O_s$.
So, $(t,g)\in O_s\times O_f=O_{(s,f)}$.
\end{proof}

The above discussion implies that for topological unosemigroups $\mathbf F=(F,\lambda_F,\rho_F)$, $\mathbf S=(S,\lambda_S,\rho_S)$ and a $\lambda_F$-respecting $\rho_S$-invertible continuous action $\alpha:F\times S\to S$ of $F$ on $S$, the triple $\mathbf S\times_\alpha\mathbf F=(S\times_\alpha F,\lambda,\rho)$ is a topological unosemigroup. This topological  unosemigroup will be called the {\em semidirect product} of the topological unosemigroups $\mathbf F$ and $\mathbf S$. By Theorems~\ref{SDL} and \ref{SDR} the topological unosemigroup $\mathbf S\times_\alpha F$ is ditopological if and only if so are the topological unosemigroups $\mathbf S$ and $\mathbf F$.

\subsection{Hartman-Mycielski Extension}

 Given a topological space $X$, for every $n \in \mathbb{N}$ by $HM_{n}(X)$ we denote the set of all functions $ f:[0,1) \rightarrow X $ for which there exists a
sequence $0 = a_{0} < a_{1} <\dots < a_{n} = 1$ such that $f$ is
constant on each interval $[a_{i}, a_{i+1}), 0 \leq i < n$. The
union $HM(X)=\bigcup_{n \in \mathbb{N}}HM_{n}(X)$ is called the
{\em Hartman-Mycielski extension} of the space $X$, see \cite{HM},
\cite{BM}.

A neighborhood sub-base of the topology of $HM(X)$ at  $f \in
HM(X)$ consists of sets $N(a, b, V, \varepsilon)$, where
\begin{enumerate}
\item $0 \leq a <b \leq 1$, $f$ is constant on $[a,b)$, $V$ is a
neighborhood of $f(a)$ in $X$, and $\varepsilon>0$, \item $g \in
N(a, b, V, \varepsilon)$ means that $|\{t \in [a, b) : g(t) \notin
V \}|<\varepsilon $, where $|\cdot|$ denotes the Lebesgue measure.
\end{enumerate}

If $X$ is a topological semigroup, then $HM(X)$ also is a
topological semigroup  with respect to the pointwise
operations of multiplication of functions. Moreover, for any continuous left (right) unit operation $u_X:X\to X$ on $X$, the unary operation $u_{HM(X)}:HM(X)\to HM(X)$, $u_{HM(X)}:f\mapsto u\circ f$, is a continuous left (right) unit operation on $HM(X)$ , see
\cite[Proposition 2]{BH}.

\begin{theorem}\label{HM} If $\mathbf X=(X,\lambda_X)$ is a ditopological left unosemigroup, then so is its Hartman-Mycielski extension $HM(\mathbf X)=(HM(X),\lambda_{HM(X)})$.
\end{theorem}

\begin{proof} For simplicity of notation, we write $\lambda$
instead of $\lambda_{HM(X)}$. Assume that $X$ is a ditopological left unosemigroup. To show
that the topological left unosemigroup $HM(X)$ is ditopological, fix any
element $f \in HM(X)$ and its sub-basic neighborhood
$O_f=N(a,b,O_{f(a)},\varepsilon)$ such that $f$ is constant on the half-interval $[a,b)\subset[0,1)$ and
$O_{f(a)}$ is an open neighborhood of $f(a)$ in $X$. Since the topological left unosemigroup $X$ is ditopological, there are neighborhoods $U_x\subset S$
and $W_{\lambda_X(x)}\subset \lambda_X(X)$ of the points $x=f(a)$ and
$\lambda_X(x)=\lambda_X(f(a))$ such that
$(W_{\lambda_X(x)}\leftthreetimes U_x)\cap
\lambda_X^{-1}(W_{\lambda_X(x)})\subset O_x$.

Consider the open neighborhoods
$U_f=N(a,b,U_{x},\frac{\varepsilon}{3})\subset HM(X)$ of $f$ and
$$W_{\lambda
(f)}=\lambda(HM(X))\cap
N(a,b,W_{\lambda_X(x)},\frac{\varepsilon}{3})\subset
\lambda(HM(X))$$ of $\lambda(f)$. We claim that each function
$g\in (W_{\lambda(f)}\leftthreetimes U_f)\cap
\lambda^{-1}(W_{\lambda(f)})\subset HM(X)$ belongs to the
neighborhood $O_f$. The inclusion $\lambda(g)\in
W_{\lambda(f)}=N(a,b,W_{\lambda_X(x)},\frac\e3)$ guarantees that
the set $A=\{t\in[a,b):\lambda_X(f(t))\notin W_{\lambda_X(x)}\}$
has Lebesgue measure $<\frac{\e}3$. On the other hand, the
inclusion $wg\in U_f=N(a,b,U_x,\frac\e3)$ implies that the set
$B=\{t\in[a,b):w(t)\cdot g(t)\notin U_x\}$ has Lebesgue measure
$<\frac\e3$. Finally, the inclusion $w\in W_{\lambda(f)}$ implies
that the set $C=\{t\in[a,b):w(t)\notin W_{\lambda_X(x)}\}$ has
Lebesgue measure $<\frac\e3$. Then the union $A\cup B\cup C$ has
Lebesgue measure $<\e$ and for each $t\in[a,b)\setminus(A\cup
B\cup C)$ we get $\lambda_X(g(t))\in W_{\lambda_X(x)}$, $ w(t)\in
W_{\lambda_X(x)}$ and $w(t)g(t)\in U_x$. Then the choice of the
neighborhoods $U_x$ and $W_{\lambda_X(x)}$ guarantees that
$g(t)\in O_{x}=O_{f(a)}$ and hence $g\in N(a,b,O_{f(a)},\e)=O_f$.
\end{proof}

By analogy we can prove:

\begin{theorem}\label{HM2} If $\mathbf X=(X,\rho_X)$ is a ditopological right unosemigroup, then so is its Hartman-Mycielski extension $HM(\mathbf X)=(HM(X),\rho_{HM(X)})$.
\end{theorem}

Theorems~\ref{HM} and \ref{HM2} imply:

\begin{theorem}\label{HM3} If $\mathbf X=(X,\lambda_X,\rho_X)$ is a ditopological unosemigroup, then so is its Hartman-Mycielski extension $HM(\mathbf X)=(HM(X),\lambda_{HM(X)},\rho_{HM(X)})$.
\end{theorem}

Since the Hartman-Mycielski space $HM(X)$ is contractible
\cite{HM}, Theorems~\ref{HM}--\ref{HM3} imply:

\begin{corollary} Each topological (left, right) unosemigroup is topologically isomorphic to (left, right) subunosemigroup of a contractible topological (left, right) unosemigroup.
\end{corollary}

\section{Acknowledgements}

The authors would like to thank Oleg Gutik for fruitful valuable remarks concerning this paper.

\end{document}